\documentclass{article}
\usepackage{fullpage}
\usepackage{amsmath}
\usepackage{amsthm}
\usepackage{amssymb}
\usepackage{url}
\usepackage{graphicx}
\usepackage{verbatim}
\usepackage{algorithm}
\usepackage{algorithmic}
\usepackage[algo2e]{algorithm2e}
\usepackage{url}
\usepackage{graphicx}
\usepackage{tikz}
\usetikzlibrary{calc}
\usetikzlibrary{patterns}
\tikzstyle{mybox} = [draw=black, fill=white,  thick,
    rectangle, inner sep=10pt, inner ysep=20pt]
\tikzstyle{mybox} = [draw=black, fill=white,  thick,
    rectangle, inner sep=2pt, inner ysep=2pt]
\newtheorem{thm}{Theorem}

\newtheorem{prop}{Proposition}
\theoremstyle{definition}

\newtheorem{remark}{Remark}

\begin{document}

\title{On the Equivalence of SDP Feasibility and a Convex Hull Relaxation for System of Quadratic Equations}
\author{Bahman Kalantari \\
Department of Computer Science, Rutgers University, NJ\\
kalantari@cs.rutgers.edu
}
\date{}
\maketitle

\begin{abstract}
We show {\it semidefinite programming} (SDP) feasibility problem is equivalent to solving a {\it convex hull relaxation} (CHR) for a finite system of quadratic equations.  On the one hand, this  offers a simple description of SDP.  On the other hand, this equivalence makes it possible to describe a version of the {\it Triangle Algorithm} for SDP feasibility based on solving CHR.  Specifically, the Triangle Algorithm either computes an approximation to the least-norm feasible solution of SDP, or using its {\it distance duality}, provides a separation when no solution within a prescribed norm exists. The worst-case complexity of each iteration  is computing the largest eigenvalue of a symmetric matrix arising in that iteration.  Alternate complexity bounds on the total number of iterations can be derived.   The Triangle Algorithm thus provides an alternative to the existing interior-point algorithms for SDP feasibility and SDP optimization.  In particular, based on a preliminary computational result, we can efficiently solve SDP relaxation of {\it binary quadratic} feasibility via the Triangle Algorithm. This finds application in solving SDP relaxation of MAX-CUT.  We also show in the case of testing the feasibility of a system of convex quadratic inequalities, the problem is reducible to a corresponding CHR, where the worst-case complexity of each iteration via the Triangle Algorithm is solving a {\it trust region subproblem}.  Gaining from these results, we discuss potential extension of CHR and the Triangle Algorithm to solving general system of polynomial equations.

\end{abstract}


\section{Introduction} \label{sec1}

In this article we prove the feasibility problem in {\it semidefinite programming} (SDP) is equivalent to a {\it convex hull relaxation} for  a finite system of quadratic equations. Specifically,  consider a system of quadratic equations $q_i(x)=b_i$, $i=1, \dots, m$, where $b_i \in \mathbb{R}$, and $q_i(x)=x^TA_ix$, $A_i$ an $n \times n$ real symmetric matrix.  The problem of testing if there is a common solution $x$ is well known to be NP-hard.  Rather than testing if the system has a solution,
we test if $b=(b_1, \dots, b_m)^T$ lies in the convex hull of the set of all $m$-tuples $(q_1(x), \dots, q_m(x))^T$, as $x$ ranges in $\mathbb{R}^n$.
We refer to this problem as the {\it convex hull relaxation} (CHR). The convex hull of a set is the smallest convex set containing that set.   By Carath\'eorory theorem, if $b$ lies in this convex hull, it can be represented as a convex combination of at most $m+1$ such $m$-tuples.  This characterization, in an elementary fashion, gives rise to a positive semidefinite matrix $X$, where for $i=1, \dots, m$, $Tr(A_i X)=b_i$. Hence $X$ is a feasible solution of an SDP. Conversely, a feasible solution to the SDP defined by these equations gives rise to a solution in the convex hull of the set of $m$-tuples.

On the one hand this gives a very simple description of SDP. On the other hand, suppose we wish to test if there is a solution $x$ to the system of quadratic equations, where $\Vert x \Vert$ is within a given radius $r$. We show how to solve the convex hull relaxation of this problem via a version of the {\it Triangle Algorithm}, a {\it fully polynomial-time approximation scheme} (FPTAS) designed to  test if a given point lies in an arbitrary given compact convex subset of the Euclidean space, \cite{kalchar, kalsep}.  In summary,  from the results in this article we gain insights on SDP, including the applicability  of the Triangle Algorithm as an alternative algorithm to interior-point algorithms for SDP. The results also offer insights on solving a general system of polynomial equation.  In the remaining of this section we give a brief review on SDP and the Triangle Algorithm.

SDP has received much attention in the literature due to its wide range of applications.  It is a generalization of LP, where the underlying nonnegativity cone is replaced with the cone of symmetric positive semidefinite matrices, see e.g.  \cite{Alizadeh, NN, van96}.  As a special case of {\it self-concordant} optimization problems, SDP can be approximated to within $\varepsilon$ tolerance in polynomial time complexity in terms of the dimensions of the problem and  $\ln 1/\varepsilon$, see \cite{NN}. The main work in each iteration of interior-point algorithms is solving a Newton system arising in that iteration.  The over all complexity in solving an SDP with $n \times n$ matrices can be as large as $O(n^{6.5} \ln 1/\varepsilon )$, see \cite{Nesterov}.  SDP relaxations have found applications in combinatorial optimization, see \cite{GW} for approximation of  the {\it MAX-CUT} problem, \cite{NW} for relaxations of nonconvex quadratic, and \cite{Lovasz} for applications in other  combinatorial problems. In some SDP relaxations the overall complexity can be reduced, e.g. to $O(n^{4.5} \ln 1/\varepsilon )$, \cite{HRVW}.

The {\it Triangle Algorithm} (TA), introduced in \cite{kalchar}, is a geometrically inspired algorithm originally designed to solve the {\it convex hull membership} problem (CHM): Test if a given $p_\circ \in \mathbb{R}^m$ lies in $conv(S)$, where $S=\{v_1, \dots, v_n \} \subset \mathbb{R}^m$.  The algorithm is endowed with {\it distant dualities} and offers fast alternative complexity bounds to polynomial-time algorithms, allowing trade-off between dependence on the dimension of the problem and the desired tolerance in approximation.  In numerical experimentations TA performs quite well.  A generalization of the Triangle Algorithm described in \cite{kalsep} tests if a  given pair of arbitrary compact convex sets $C, C' \subset \mathbb{R}^m$ intersect, or if they are separable. Specifically, it can solve any of the following four problems when applicable: (1)
computing an approximate point of intersection, (2) computing a separating hyperplane, (3) computing an approximation to the optimal pair of supporting hyperplanes, (4) approximating the distance between the sets.  In particular, CHM and the {\it hard margin} problem (SVM) are very special cases. In the general version of the Triangle Algorithm the complexity of each iteration depends on the nature and description of the underlying convex sets.  In this article we are interested in  the version of the algorithm where one of the sets is a singleton point. We refer to {\it General-CHM} as the problem of testing if a given point $p_\circ \in \mathbb{R}^m$ lies in a given compact convex subset $C$ of $\mathbb{R}^m$ which may be given as $conv(S)$, the convex hull of some compact subset $S$.

The organization of the article is as follows.  Section \ref{sec2}, describes the equivalence of SDP feasibility and the convex hull relaxation for a system of quadratic equations. Section \ref{sec3}, describes the Triangle Algorithm and its complexity for the General-CHM. Section \ref{sec4}, specializes the Triangle Algorithm for solving the convex hull relaxation, as well as related problems. We end with concluding remarks.

\section{Equivalence of SDP Feasibility and Convex Hull Relaxation} \label{sec2}

Let $\mathbf{S}=\{A_1, \dots, A_m\}$ be a subset of  $\mathbb{S}^n$,  the set of $n \times n$ real symmetric matrices. Let
$b =(b_1, \dots, b_m)^T \in \mathbb{R}^m$, $b \not = 0$.  The Frobenius inner product of $X,Y \in \mathbb{S}^n$ is denoted by any of the following equivalent notations
\begin{equation}  \label{eq2}
\langle X, Y \rangle_F=Tr(XY)=X \bullet Y = \sum_{i=1}^n \sum_{j=1}^n x_{ij} y_{ij}.
\end{equation}
We write $X \succeq 0$ for $X \in \mathbb{S}_+^n$, the cone of positive semidefinite matrices in $\mathbb{S}^n$. The {\it SDP feasibility} problem is testing if $b \in \mathbf{P}$, where
\begin{equation} \label{eq1}
\mathbf{P}=\{\mathbf{A}(X) \equiv (A_1 \bullet X, \dots, A_m \bullet X )^T: \quad X \succeq 0\}.
\end{equation}
For $i=1, \dots, m$, let $q_i(x)=x^TA_ix$. Let $Q(x)=(q_1(x), \dots, q_m(x))^T$.
Consider testing the solvability of the system of quadratic equations $Q(x)=b$. When each $A_i$ is a diagonal matrix solving this system is equivalent to solving a linear programming feasibility problem (replacing squared variables as nonnegative ones). However, it is NP-hard in general.  Rather than solving this NP-hard problem we consider the {\it convex hull relaxation} (CHR) of the problem defined as testing if $b$ lies in $\mathbf{C}$, where
\begin{equation} \label{eq4}
\mathbf{C}= conv \big (\{Q(x)=(q_1(x), \dots, q_m(x))^T:  x \in \mathbb{R}^n\} \big )=
\big \{ \sum_{i=1}^t \alpha_i Q(x_i): \sum_{i-1}^t \alpha_i =1, \alpha_i \geq 0,  x_i  \in \mathbb{R}^n \big \},
\end{equation}
Given a real number $r$, let
\begin{equation} \label{eq5}
\mathbf{C}(r)= \big \{ \sum_{i=1}^t \alpha_i Q(x_i) \in \mathbf{C}: \Vert x_i \Vert \leq r \big \}, \quad
\mathbf{P}(r)= \big \{ \mathbf{A}(X) \in \mathbf{P}: Tr(X) \leq r^2 \big \}.
\end{equation}

\begin{thm}  \label{thm1} $b \in \mathbf{C}(r)$ if and only if $b \in \mathbf{P}(r)$.  In particular,
$b \in \mathbf{C}$ if and only if $b \in \mathbf{P}$.
\end{thm}
\begin{proof} Assume $b \in \mathbf{C}(r)$. From the definition of $\mathbf{C}(r)$ and Carath\'eodory theorem, for some $t \leq m+1$, there exists $x_1, \dots, x_t \in \mathbb{R}^n$, each $\Vert x_i \Vert \leq r$, such that
\begin{equation} \label{eq6}
\sum_{i=1}^t \alpha_i Q(x_i) = b, \quad \sum_{i=1}^t \alpha_i =1, \quad \alpha_i \geq 0, \quad \forall i.
\end{equation}
Set $X_i= x_i {x_i}^T$, $i=1, \dots, t$. Then $X_i \in \mathbb{S}^n_+$.  Hence
$X = \sum_{i=1}^t \alpha_i X_i \in  \mathbb{S}^n_+$. Moreover,  for each $k=1, \dots, m$,
\begin{equation} \label{eq7}
A_k \bullet X  = A_k \bullet \sum_{i=1}^t \alpha_i X_i= \sum_{i=1}^t \alpha_i A_k \bullet X_i = \sum_{i=1}^t \alpha_i A_k \bullet x_i {x_i}^T= \sum_{i=1}^t \alpha_i q_k(x_i) =  \sum_{i=1}^t \alpha_i b_k=b_k.
\end{equation}
Additionally,
\begin{equation} \label{eqbb}
Tr(X)=\sum_{i=1}^t \alpha_i Tr(x_ix_i^T)
= \sum_{i=1}^t \alpha_i \Vert x_i \Vert^2  \leq \sum_{i=1}^t \alpha_i  r^2 =r^2.
\end{equation}
Hence $b \in \mathbf{P}(r)$. Conversely, suppose $b \in \mathbf{P}(r)$. Then there exists $X \in \mathbb{S}^n_+$ such that $A_k \bullet X=b_k$, $k=1, \dots, m$, $Tr(X) \leq r^2$.
Let the spectral decomposition of $X$ be $X=U \Lambda U^T$, with $\Lambda={\rm diag}(\lambda_1, \dots, \lambda_n)$, $U=[u_1, \dots, u_n]$ corresponding to eigenvalue-eigenvectors, $\Vert u_i \Vert =1$, for all $i$.  We have $Tr(X)=\sum_{i=1}^{n} \lambda_i$. For $i=1, \dots, n$, let  $\alpha_i= \lambda_i/ Tr(X)$, $x_i= \sqrt{Tr(X)} u_i$. Then $\Vert x_i \Vert = \sqrt{Tr(X)}$.  Thus $X$ is a convex combination of $X_i=x_ix_i^T$, $i=1, \dots, n$ since we have:
\begin{equation} \label{eq8a}
X = \sum_{i=1}^n \lambda_i u_i u_i^T = \sum_{i=1}^n \alpha_i x_ix_i^T, \quad
\sum_{i=1}^n \alpha_i = (\sum_{i-1}^n \lambda_i)/Tr(X)=1,  \quad \alpha_i \geq 0, \forall i.
\end{equation}
Thus for $k=1, \dots, m$,
\begin{equation} \label{eq8}
A_k \bullet X = A_k \bullet \sum_{i=1}^n \lambda_i u_i u_i^T = \sum_{i=1}^n \alpha_i A_k  \bullet x_ix_i^T= \sum_{i=1}^n \alpha_i x_i^TA_k x_i = \sum_{i=1}^n \alpha_i q(x_i)= b_k.
\end{equation}
Also, $\Vert x_i \Vert \leq \sqrt{Tr(X)} \leq r$. Hence $b \in \mathbf{C}(r)$.
\end{proof}
\begin{remark} To prove the converse we could have used any Cholesky factorization of $X$ instead of its spectral decomposition. This will be more efficient.
\end{remark}

\begin{remark} Suppose the quadratic system $Q(x)=b$ is inhomogeneous, say for $i=1, \dots, m$, $q_i(x)=x^TA_ix+c_i^Tx+d_i$, where $c_i \in \mathbb{R}^n$, $d_i \in \mathbb{R}$. We can convert $q_i(x)$ to a homogeneous one as follows. First, the constant term $d_i$ is absorbed in $b_i$.  If $c_i \not =0$ for some $i$, we introduce a new variable $z$ and  replace $q_i(x)$ with $q_i(x,z)=x^TA_ix+z c_i^Tx$.  Also we add a new equation $q_{m+1}(x,z)=z^2=b_{m+1}=1$.  Now suppose $(x,z)$ satisfies
\begin{equation} \label{eq11}
Q(x,z) \equiv (q_1(x,z), \dots, q_{m+1}(x,z))^T= (b_1, \dots, b_{m+1})^T.
\end{equation}
If $z=1$, then $Q(x)=b$.  Otherwise, $z=-1$ but then $(-x,-z)$ also satisfies the above equation so that $Q(-x)=b$.
\end{remark}

\begin{remark} The set of $x_i$'s in any convex combination satisfying  $\sum_{i=1}^t  \alpha_i Q(x_i)=b$ induces approximate solutions to $Q(x)=b$. We can improve them by replacing each $x_i$ with $Q(x_i) \not =0$, with $\sqrt{\gamma^*_i} x_i$, where $\gamma^*_i={\rm argmin} \{\Vert \gamma Q(x_i)- b \Vert:  \gamma \in \mathbb{R} \}={b^TQ(x_i)}/{\Vert Q(x_i) \Vert^2}$.
\end{remark}

\section{Triangle Algorithm for General Convex Hull Membership} \label{sec3}
The  {\it general convex hull membership} (General-CHM) refers to the following problem: Given a point $p_\circ \in \mathbb{R}^m$, and a compact subset $S$ of $\mathbb{R}^m$, test if $p_\circ$ lies in $C=conv(S)$, the convex hull of $S$. If $S$ is also convex, then $C=S$.
Given $\varepsilon \in (0,1)$, the Triangle Algorithm either computes $p' \in C$ such that $\Vert p'-p_\circ \Vert \leq \varepsilon$, or  finds a hyperplane that separates $p_\circ$ from $C$. The Triangle Algorithm works as follows:  Given arbitrary {\iterate} $p' \in C$, while $\Vert p' - p_\circ \Vert > \varepsilon$,  tests if there exists a {\it pivot}, i.e. $v \in C$ that satisfies the following equivalent conditions
\begin{equation} \label{eqaa}
\Vert p' - v \Vert \geq \Vert p_\circ-v \Vert, \quad (p_\circ -p')^Tv \geq \frac{1}{2} ( \Vert p_\circ \Vert^2 - \Vert p' \Vert^2).
\end{equation}
If no pivot exists, $p'$ is called a {\it witness} since it provides a hyperplane  $H$ that separates $p_\circ$ from $C$, proving $p_\circ \not \in C$. Specifically, $H$ is the orthogonal bisector of the line segment $p_\circ p'$:
\begin{equation} \label{eq12}
H=\{x: (p'-p_\circ)^Tx = \frac{1}{2}(\Vert p' \Vert^2 -\Vert p_\circ \Vert^2)\}.
\end{equation}

A {\it strict pivot} is a pivot where $\angle p'p_\circ v \geq \pi/2$. Equivalently, $(p'-p_\circ)^T (v-p_\circ) \leq 0$. That is,
\begin{equation} \label{eq13}
(p_\circ - p')^Tv \geq \Vert p_\circ \Vert^2 - p'^Tp_\circ.
\end{equation}

When $p'$ admits a pivot or a strict pivot $v$, we can get closer to $p_\circ$ by replacing the iterate $p'$ with the closest point to $p_\circ$ on the line segment $p'v$. From the definition of strict pivot and since a compact convex set is the
convex hull of its extreme points (see e.g. Krein-Milman theorem) we have

\begin{prop} \label{eq14} Let $c= p_\circ - p'$.  Then $v \in C$ is a strict pivot if and only if
\begin{equation}
\max\{c^Tx: x \in C\}= \max \{c^Tx: x \in S\}  \geq c^Tv \geq \Vert p_\circ \Vert^2 - p_\circ^Tp.
\end{equation}
\end{prop}

From Proposition \ref{eq14}, if there is strict pivot we can compute it via a maximization  so that Triangle Algorithm can be described as follows:

\begin{algorithm}[H]
\SetAlgoNoLine
\KwIn{$S \subset \mathbb{R}^m$, $p_\circ \in \mathbb{R}^m$, $\varepsilon \in (0,1)$.}
\normalsize
Initialization: Select arbitrary $p' \in S$.
\vskip 0.01cm
\While{$\Vert p_\circ  - p' \Vert > \varepsilon$}{$v  \gets {\rm argmax} \{(p_\circ - p')^Tx:x \in conv(S)\}$
\vskip 0.01cm
\lIf{$(p_\circ - p)^Tv < \Vert p_\circ \Vert ^2 - p_\circ^T p$}{Stop, $p_\circ \not \in conv(S)$} \Else{ $\alpha \gets (p_\circ-p')^T(v-p') /{\Vert v - p' \Vert^2}$,

$p' \gets (1- \alpha)p'+ \alpha v$}}
\vskip 0.01cm
    \caption{Triangle Algorithm for General-CHM}
\end{algorithm}

The correctness of the Triangle Algorithm is due to the following (see \cite{kalchar, kalsep}):
\begin{thm}  \label{thm2} {{\rm (Distance Duality)}}
$p_\circ \in conv(S)$ if and only if for each  $p' \in conv(S)$ there exists a (strict) pivot $v \in conv(S)$. Equivalently,  $p_\circ \not \in C$ if and only if there exists a witness $p' \in C$. \qed
\end{thm}
The iteration complexity  for Triangle Algorithm is given in the following theorem.

\begin{thm}   \label{thm3}  {\rm (Complexity Bounds)} Let $R= \max \{ \Vert x- p_\circ \Vert:  x \in conv(S)\}$.  Let $\varepsilon \in (0,1)$.

(i): Triangle  in $O(1/\varepsilon^2)$ iterations either computes
$p_\varepsilon \in conv(S)$ with $\Vert p_\circ - p_\varepsilon \Vert \leq R \varepsilon$, or a witness. In particular, if  $p_\circ \not \in conv(S)$ and $\delta_*= \min\{\Vert x-p_\circ \Vert: x \in conv(S)\}$, the number of iterations to compute a witness is
$O(R^2/\delta^2_*)$.  Furthermore, given  any witness $p' \in conv(S)$, we have
\begin{equation} \label{eq15}
\delta_* \leq \Vert p' -  p_\circ \Vert \leq 2 \delta_*,
\end{equation}
i.e. a witness estimates the distance to $conv(S)$ to within a factor of two.

(ii) Suppose a ball of radius $\rho >0$ centered at $p_\circ$ is contained in the relative interior of $conv(S)$. If Triangle Algorithm uses a strict pivot in each iteration, it computes $p_\varepsilon \in conv(S)$ satisfying $\Vert p_\varepsilon - p_\circ \Vert  \leq \varepsilon$ in $O\big ((R/\rho)^{2} \ln {1}/{\varepsilon} \big )$ iterations.
\qed
\end{thm}
\begin{remark} When case (ii) is applicable, we can think of the ratio $R/\rho$ as a {\it condition number} for the problem.  If this condition number is not large the complexity is only logarithmic in $1/\varepsilon$.
\end{remark}

In a more elaborate fashion it can be shown that Triangle Algorithm can approximate the distance from $p_\circ$ to $C$ to within any prescribed accuracy $\varepsilon$,  see \cite{kalsep}.

\section{Triangle Algorithm for Convex Hull Relaxation} \label{sec4}
Consider testing if $b=(b_1, \dots, b_m)^T$ lies on
$\mathbf{C}(r)=conv(\{Q(x)= (q_1(x), \dots, q_m(x))^T:  \Vert x \Vert \leq r \})$,  where
$q_i(x)=x^TA_ix$,  $A_i \in \mathbb{S}^n$, $i=1, \dots, m$, and $r >0$ a given number.
We test this via the Triangle Algorithm for General-CHM, where $p_\circ=b$ and $C=\mathbf{C}(r)$. Given $b' \in \mathbf{C}(r)$, let $c=(b-b')=(c_1, \dots, c_m)^T$, $A= \sum_{i=1}^m c_i A_i$. Then from Proposition \ref{eq14}
to test if there is a strict pivot, in the worst-case, amounts to solving the following optimization problem:
\begin{equation} \label{eq18}
\max \{c^T Q(x):  \Vert x \Vert \leq r\} =\max \{x^TAx:  \Vert x \Vert  \leq r\} = r^2 \lambda_{\max}= r^2 u_{\max}^TAu_{\max}=r^2A \bullet  u_{\max} u_{\max}^T.
\end{equation}
where $\lambda_{\max}$ is the largest eigenvalue of $A$ and $u_{\max}$ the corresponding unit eigenvector. The computation of  $\lambda_{\max}$ and $u_{\max}$
can be achieved via the well-known {\it Power Method}.

If initially we have an estimate $r$ of the norm of a possible solution in $\mathbf{C}$, we start with that estimate and test if $b \in \mathbf{C}(r)$.
If $b \not \in \mathbf{C}(r)$, Triangle Algorithm will eventually compute  a witness. In that case we replace $r$ with $2r$ and repeat the above until we have computed an approximate solution or $r$ exceeds an upper bound $r_{\max}$. We first compute an initial lower bound:

\begin{prop} \label{prop1}  Let $r_\circ=   \min \big \{\sqrt {{|b_k|}/{\Vert A_k \Vert}}: k=1, \dots, m, b_k \not =0 \big \}.$ Then for any $r < r_\circ$, $b  \not \in \mathbf{C}(r)$.
\end{prop}
\begin{proof} Since $b \not =0$, $r_\circ$ is well defined.
Suppose $b \in \mathbf{C}$. Then  there exists $t$ such that
for each $i=1, \dots, t$, $b_k= \sum_{i=1}^t \alpha_i x_i^TA_k x_i$,  $\sum_{i=1}^t \alpha_i=1$, $\alpha_i \geq 0$. This implies the following from which the proof follows:
\begin{equation} \label{eq20}
|b_k| =  \sum_{i=1}^t \alpha_i |x_i^TA_k x_i| \leq \sum_{i=1}^t \alpha_i \Vert A_k \Vert \Vert x_i \Vert^2 \leq \Vert A_k \Vert  \max \{ \Vert x_i \Vert^2: i=1, \dots, t\}.
\end{equation}
\end{proof}

The algorithm is described next.

\begin{algorithm}[H]
\SetAlgoNoLine
		\caption{Triangle Algorithm for SDP Feasibility}
\KwIn{$S=\{A_1, \dots, A_m\} \subset \mathbb{S}^n$,
$b \in \mathbb{R}^m$, $r_\circ >0$, $r_{\max}$, an upper bound on $r$, $\varepsilon \in (0,1)$}
\normalsize
Initialization: Pick arbitrary $x' \in \mathbb{R}^n$,  $\Vert x' \Vert \leq r_\circ$.
$r \gets r_\circ$,  $X' \gets x'x'^T$, $b' \gets \mathbf{A}(X') \equiv (A_1 \bullet X', \dots, A_m \bullet X' )^T$.
\vskip 0.01cm
\While{$\Vert b  - b' \Vert > \varepsilon$ and $r \leq r_{\max}$}{$A \gets \sum_{i=1}^m(b_i-b'_i) A_i$,  $\lambda_{\max} \gets$  max eigenvalue of $A$,
$u_{\max} \gets$ corresponding unit eigenvector,
\vskip 0.01cm
$V \gets u_{\max}u_{\max}^T$, $v \gets Q(r u_{\max})= \mathbf{A}(V)$
\vskip 0.01cm
\lIf{$r^2 \lambda_{\max} < \Vert b \Vert^2 - b^T b'$}{$r \gets 2r$} \Else{$\alpha =(b-b')^T(v-b') /{\Vert v - b' \Vert^2}$,

$b' \gets (1-\alpha)b' + \alpha v$,

$X' \gets (1-\alpha)X' + \alpha V$}}
\vskip 0.01cm
\end{algorithm}

\begin{remark}
Ignoring the complexity of the Power Method, each iteration takes $O(mn^2)$, mainely  to evaluate $\mathbf{A}(V)$. Also note that we can forgo computing $A$ explicitly in each iteration because we only need matrix-vector multiplication $A w$ but this is the sum of $c_kA_kw$ over $k=1, \dots, m$. We may not need to run the Power Method to optimality in each iteration because all is needed is a pivot or strict pivot. Thus after each iteration of the Power Method we can easily check if we have reached the appropriate condition.
\end{remark}

\begin{thm} \label{thm4} {\rm (Complexity of Solving CHR)}
Let $r_\circ$ be as in Proposition \ref{prop1}. Given $r \geq r_\circ$ let
\begin{equation}
R_r= \max \{\Vert Q(x) - b \Vert :  r_\circ \leq \Vert x \Vert \leq 2r\}.
\end{equation}

(i) Suppose $b \in \mathbf{C}$.  Let $r_*$ be the smallest $r$ for which  $b \in \mathbf{C}(r)$.  Then in $O(1/\varepsilon^2)$ iterations the Triangle Algorithm  computes $x_i \in \mathbb{R}^n$, $i=1, \dots, t$,  $t \leq m+1$, where $\sum_{i=1}^t \alpha_i Q(x_i) \in \mathbf{C}(2r_*)$,  and if $X= \sum_{i=1}^t \alpha_i x_i x_i^T$,  $\mathbf{A}(X) = (A_1 \bullet X, \dots, A_m \bullet X)^T \in
\mathbf{P}(2r_*)$,  satisfying
\begin{equation} \label{eq21}
\Vert \sum_{i=1}^t \alpha_i Q(x_i) -b \Vert =\Vert \mathbf{A}(X) - b \Vert  \leq \varepsilon R_{r_*}.
\end{equation}

(ii) Given $r >0$, suppose the relative interior of  $\mathbf{C}(r)$ contains
the ball of radius $\rho >0$ centered at $b$, i.e.
\begin{equation}
\mathbf{C} \cap \{y \in \mathbb{R}^m: \Vert y - b \Vert \leq \rho \} \subset \mathbf{C}(r).
\end{equation}
The algorithm in $O((R_r/\rho)^2 \ln 1/\varepsilon)$ iterations computes $x_i \in \mathbb{R}^n$, $i=1, \dots, t \leq m+1$, where $\sum_{i=1}^t \alpha_i Q(x_i)$ lies in $\mathbf{C}(r)$,  and if  $X= \sum_{i=1}^t \alpha_i x_i x_i^T$, then $\mathbf{A}(X)$ lies in  $\mathbf{P}(r)$, satisfying
\begin{equation}
\Vert \sum_{i=1}^t \alpha_i Q(x_i) -b \Vert=\Vert \mathbf{A}(X) - b \Vert  \leq \varepsilon.
\end{equation}

(iii) Suppose $b \not \in \mathbf{C}(r_{\max})$.  Let $\delta$ be the distance from $b$ to $\mathbf{C}(r_{\max})$. Then in $O(r_{\max}^2/\delta^2)$ iterations the algorithm computes a witness $b' \in \mathbf{C}(r_{\max})$, i.e. orthogonal bisector of $bb'$ separates $b$ from $\mathbf{C}(r_{\max})$.
\end{thm}
\begin{proof} The proof of complexity theorem mainly follows from  Theorem \ref{thm3}. To justify (i),  note that for any $b'=\sum_{i=1}^t \alpha_i Q(x_i) \in \mathbf{C}(r)$, $\Vert b' - b \Vert \leq  \sum_{i=1}^t \alpha_i \Vert Q(x_i) -b \Vert \leq \max \{\Vert Q(x_i) - b \Vert$.  Hence the quantity $R_r$ corresponds to $R$ defined in Theorem \ref{thm3}. Each time a witness is calculated with respect to a current $r$ we are not able to reduce the gap between $b$ and the current iterate $b'$. However, starting at $r_\circ$ as the initial value, we will  double the current value of $r$ at most $O(\ln (r_*/r_\circ))$ times. On the other hand, each time a witness is computed we can increase the value of $r$ until the current iterate $b'$ admits a pivot in which case the gap between $b$ and $b'$ will decrease. The above argument justifies (i).
Proof of (ii) and (iii) also follow from Theorem \ref{thm3}, keeping Theorem \ref{thm1} in view.
\end{proof}

\begin{remark} When $Q(x)=b$ is solvable, the Triangle Algorithm computes a relaxed solution, i.e. a set of points $x_i \in \mathbb{R}^n$, $i=1, \dots, t$, $t \leq m+1$, such that $\sum_{i=1}^{t} \alpha_i Q(x_i)=b$, $\sum_{i=1}^t \alpha_i=1$, $\alpha_i \geq 0$. There maybe  cases where $Q(x)=b$ is unsolvable but the relaxation is solvable. However, when $\mathbf{C}(r)$ is empty Triangle Algorithm computes a witness, implying  $Q(x)=b$ is not solvable. This is an important feature of the Triangle Algorithm, the ability to produce in some cases a certificate to lack of solvability of a quadratic system.
\end{remark}

\begin{remark}  From the implementation point of view,  there are many fine-tuning steps that can be taken to improve the efficiency of the Triangle Algorithm. For instance, we may not need to compute in each iteration the largest eigenvalue in (\ref{eq18}). If we store the pivots, or a subset of them, as they are generated, there is a good chance that these can be used one or more times in subsequent iterations.  Such discussions are  described in \cite{kal2019}, where a Triangle Algorithm is proposed for an SDP version of the {\it convex hull membership} (CHM), called {\it spectrahull membership} (SHM). Given substantial computational experiences with the Triangle Algorithm for CHM, and even for the generation of all vertices of the convex hull of a finite set of point, see \cite{AKZ}, and preliminary computation with binary quadratic feasibility, we would expect the iteration complexity to be quite reasonable for solving SDP feasibility, as well as SDP optimization problems, considered in next remark.
\end{remark}

\begin{remark}  Consider the SDP optimization problem, $\max \{A_0 \bullet X:
\mathbf{A}(X) = b,  Tr(X) \leq r^2, X \succeq 0\}$.  We first test the feasibility of $\mathbf{C}(r)$ via the Triangle Algorithm.  This produces a set, $x_i \in \mathbb{R}^n$, $i=1, \dots, t \leq m+1$ such that $\sum_{i=1}^{t} \alpha_i Q(x_i)=b$ (or approximate equality), $\sum_{i=1}^t \alpha_i=1$, $\alpha_i \geq 0$.  With $q_0(x)=x^TA_0x$,  let $b_0=\sum_{i=1}^{t} \alpha_i q_0(x_i)$. Assume the maximum objective value is bounded.  We increase the value of $b_0$ to a larger value, $b'_0$, and test via the Triangle Algorithm if $\overline b =(b'_0, b_1, \dots, b_m)^T$ lies in the augmented convex hull relaxation $\mathbf{ \overline C}(r)= conv (\{\overline Q(x)=(q_0(x), q_1(x), \dots,q_m(x))^T : \Vert x  \Vert \leq r \})$.  A good strategy is to pick $b_0'$ to be the largest value so that $\overline b$ admits a strict pivot in $\mathbf{ \overline C}(r)$ with respect to $b=(b_0, b_1, \dots, b_m)^T$. Then we continue with the Triangle Algorithm until we either get a witness, in which case we  decrease the value of $b_0'$ to a value that admits  a strict pivot for the new $\overline b$,  or we get sufficiently close to $\overline b$, in which case we replace $b$ with $\overline b$ and repeat the above strategy. In other words, it is possible to approximate the optimal objective value as close as we wish, essentially by solving several interrelated feasibility problems. Many strategies can be described and tested.  If the SDP optimal objective is unbounded we have to terminate the algorithm at some point, as would any other algorithm for the problem.
\end{remark}

\subsection{CHR and SDP Relaxation for  Binary Quadratic Feasibility}
Given $A \in \mathbb{S}^n$, and a real number $\alpha$, the {\it binary quadratic feasibility} is testing the feasibility of
\begin{equation} \label{CH21qpzone}
\big \{x \in \mathbb{R}^n: x^TAx= \alpha,  \quad x_i \in \{-1, 1\} \big\}.
\end{equation}
The optimization version of the problem is to find the maximum value of $\alpha$ for which the above is feasible.  The SDP  relaxations for (\ref{CH21qpzone}) is testing the feasibility of
\begin{equation} \label{CH21qpzoneR}
\{X \in \mathbb{S}^n: Tr(A \bullet X)= \alpha,  \quad x_{ii}=1, i=1, \dots, n, \quad X \succeq 0  \}.
\end{equation}

An example of binary quadratic optimization is the Max-CUT problem for which Goemans and Williamson \cite{GW} showed that once the SDP optimization is solved, its optimal solution $X^*$, given rise to an approximate solution:  First computing a Cholesky factorization, $X^*=V^TV$, where $V=[v_1, \dots, v_s]$, $v_i \in \mathbb{R}^s$, $s$ the rank of $X^*$, then by choosing a random hyperplane through the origin in $\mathbb{R}^s$, $p^Tx= 0$,  and assigning  $x_i=1$ if $p^Tv_i >0$, $x_i=-1$ if $p^Tv_i <0$.  Their scheme results in an approximation to within $.878$ of optimality.  For the general symmetric matrix $A$ the quality of this approximation cannot be guaranteed.

In solving CHR corresponding to binary quadratic feasibility via the Triangle Algorithm, $b=(\alpha, 1, \dots, 1) \in \mathbb{R}^{n+1})$, and if $b'=(b'_0, b'_1, \dots, b'_n) \in \mathbb{R}^{n+1})$, it follows that the main work in each iteration is computing the largest eigenvalue of the matrix $c_0A + {\rm diag}(c_1, \dots, c_n)$, where $c_i=b_i-b'_i$. Much is known about the eigenvalues of such matrices which help estimating the largest eigenvalue.  In a preliminary testing with the Triangle Algorithm for CHR corresponding to binary quadratic, problems with size of matrix up to $n=1000$ are solved in a few seconds. We will report details of computational results elsewhere.

\subsection{A CHR for Feasibility of System of Convex Quadratic Constraints}
Consider testing the feasibility of a system of convex quadratic inequalities  $x^TA_kx + c_k^T x \leq b_k, k=1, \dots, m$, where $A_k \in \mathbb{S}^n_+$.
Such problem can be handled via self-concordance theory, see  \cite{Nesterov}, however  Newton systems need to be solved and thus it may be desirable to offer trade-off between the number of iterations and the complexity of each iteration. The Triangle Algorithm can be used to solve this feasibility problem as a General-CHM and a modified CHR. Adding slack variables, the above system is feasible if and only if the inhomogeneous quadratic equation $Q(z)=b$ is solvable, where
$$Q(z)=(q_1(z), \dots, q_m(z))^T, \quad  q_k(z)= x^TA_kx + c_k^T x+s_k^2, \quad k=1, \dots, m, \quad  b=(b_1, \dots, b_m)^T.$$
Using convexity of $q_k(z)$'s, it is straightforward to show
\begin{thm} $Q(z)=b$  $\iff$  $b$ lies in
$conv \big (\{ Q(z): z \in \mathbb{R}^{n+m}\} \big )$. \qed
\end{thm}
This implies testing if $Q(z)=b$ is solvable is equivalent to an inhomogeneous  CHR, yet solvable via the Triangle Algorithm, where  in the worst-case each iteration solves a {\it trust region subproblem}: $\max \{z^T A z+c^Tz: \Vert z \Vert \leq r\}$, $A \in \mathbb{S}^n$, an well-known optimization  problem, see e.g. \cite{Rojas}.

\section*{Concluding Remarks}
By employing basic results from convexity, we have shown the equivalence of the solvability of a convex hull relaxation (CHR) to a system of quadratic equation, and the general SDP feasibility problem. On the one hand this demonstrates in an elementary fashion how an SDP feasibility problem may arise, also offering insights on SDP itself.
Despite the tremendous literature on SDP, in particular SDP relaxation of quadratic programming optimization problems (e.g. \cite{GW,Luo}),
the equivalence shown in this article appears not to have been given previously.  On the other hand, the significance of this equivalence becomes evident by having described  a version of the Triangle  Algorithm that solves the convex hull relaxation of a system of quadratic equations and at the same time offers a new algorithm for SDP feasibility and optimization.
Each iteration of the Triangle Algorithm requires computing a pivot which in the worst-case amounts to computing the largest eigenvalue of a symmetric matrix arising in that iteration and can thus be achieved via the Power Method.  Also, pivots can be stored and reused so that the largest eigenvalue computation is not necessary in each iteration.
The Triangle Algorithm thus offers an alternative to the interior-point algorithms for SDP whose iterations need to solve a Newton system which could be expensive for large size problems. The Triangle Algorithm is also equipped with a duality of its own (distance duality) that provides new insights on SDP. On the other hand, the insights gained from SDP can help in solving a system of quadratic equations. For instance, in ways in which we may convert the solution of the convex hull relaxation to a positive semidefinite matrix and then converting it back to an approximate solution to the quadratic system itself.
Goemans-Williamson
procedure for converting the optimal solution to the SDP relaxation of Max-Cut problem is one such approach. Based on a preliminary computational result with {\it binary quadratic} feasibility, a problem that has application to MAX-CUT,
the Triangle Algorithms solves the corresponding CHR in a few seconds for
matrices of dimension up to $1000$. In this article we have also justified that testing the feasibility of a system of convex quadratic inequalities is equivalent to an inhomogeneous  CHR, yet solvable via the Triangle Algorithm, where  in the worst-case each iteration solves a trust region subproblem.  The results in the article show the versatility and utility of the Triangle Algorithm in solving variety of problems: binary quadratic, MAX-CUT, SDP feasibility and optimization, and convex quadratic feasibility.  In forthcoming work we will report on computational results and comparison with existing algorithms.

Viewing the connection between convex hull relaxation (CHR) of a system of quadratic equations, SDP feasibility problem,  and the Triangle Algorithm suggests
that extension of these may be possible when considering a system of homogeneous polynomial equations in several variables.  Consider solving  a system of equations
$Q(x) =b$, where $Q(x)=(p_1(x), \dots, p_m(x))^T$,  each $p_i(x)$ a
homogeneous polynomials of  degree $d \geq 1$.  Consider the relaxation that tests if $b \in conv(\{Q(x): x \in \mathbb{R}^n\})$.  Firstly,  it is possible an analogous Theorem \ref{thm1} can be stated where the $n \times n$ matrices $A_i$ are replaced with multidimensional symmetric matrices.  Furthermore, we can then attempt to solve the corresponding  relaxation via the Triangle Algorithm. However,  the success of such method relies on the computation of a pivot. For instance, if each $p_i(x)$ is a cubic polynomial, the computation of a pivot, in the worst case, amounts to computing the maximum of a trilinear form over the unit ball.  It goes without saying that the problem of solving a general system of polynomial equations is a profound problem with deep underlying mathematics, see e.g. \cite{Sturmfels, Parrilo}. It is no easy task to solve a system of polynomial equations.  In particular, the sophisticated Gr\"obner basis method is very limited in the size of the problems it can solve.  The simplicity of the Triangle Algorithm and the results described in this article give rise to the this question: can we solve the convex hull relaxation of the general system of homogeneous polynomial equations via the Triangle Algorithm? Such algorithm would need to compute the maximum of a symmetric multidimensional matrix over the unit ball. The problem of computing eigenvalues of tensors via high-order Power Method  has been addressed, see e.g. \cite{Tensor}.  It may thus not be too far fetched to consider solving the convex hull relaxation via the Triangle Algorithm, given that the homogeneous degree is small.  Further research is of course necessary, complemented with computational experimentation. We will report on experiment with quadratic systems. We mention in passing that when the homogeneous degree $d=1$,  the corresponding Triangle Algorithm gives rise to a new iterative algorithm for solving a linear system, \cite{KLZ}. Our computational results with this, establishes the Triangle Algorithm as a new iterative method that is very competitive with the existing iterative algorithms for linear systems.

\bigskip

\end{document}